\documentclass[11pt]{amsart}

\newtheorem{theorem}{Theorem}[section]

\newtheorem{corollary}[theorem]{Corollary}

\newtheorem{lemma}[theorem]{Lemma}

\newtheorem{proposition}[theorem]{Proposition}

\theoremstyle{definition}
\newtheorem{remark}[theorem]{Remark}

\usepackage[colorlinks, bookmarks=true]{hyperref}
\usepackage{color,graphicx,shortvrb}

\usepackage{enumerate}
\usepackage{amsmath}
\usepackage{amssymb}

\usepackage[latin 1]{inputenc}

\def\11{\textbf{$1$}}

\begin{document}

\numberwithin{equation}{section}

\title[A Kowalski-S{\l}odkowski theorem for von Neumann algebras]{A Kowalski-S{\l}odkowski theorem for 2-local $^*$-homomorphisms on von Neumann algebras}

\author[Burgos]{Maria Burgos}
\email{maria.burgos@uca.es}
\address{Departamento de Matematicas, Facultad de Ciencias Sociales y de la Educacion,  Universidad de Cadiz, 11405, Jerez de la Frontera, Spain.}

\author[Fern\'{a}ndez-Polo]{Francisco J. Fern\'{a}ndez-Polo}
\email{pacopolo@ugr.es}
\address{Departamento de An{\'a}lisis Matem{\'a}tico, Facultad de
Ciencias, Universidad de Granada, 18071 Granada, Spain.}

\author[Garc{\' e}s]{Jorge J. Garc{\' e}s}
\email{jgarces@correo.ugr.es}
\address{Departamento de An{\'a}lisis Matem{\'a}tico, Facultad de
Ciencias, Universidad de Granada, 18071 Granada, Spain.}

\author[Peralta]{Antonio M. Peralta}
\email{aperalta@ugr.es}
\address{Departamento de An{\'a}lisis Matem{\'a}tico, Facultad de
Ciencias, Universidad de Granada, 18071 Granada, Spain.}
\curraddr{Visiting Professor at Department of Mathematics, College of Science, King Saud University, P.O.Box 2455-5, Riyadh-11451, Kingdom of Saudi Arabia.}

\thanks{Authors partially supported by the Spanish Ministry of Science and Innovation,
D.G.I. project no. MTM2011-23843, and Junta de Andaluc\'{\i}a grant FQM375.
The fourth author extends his appreciation to the Deanship of Scientific Research at King Saud University (Saudi Arabia) for funding the work through research group no. RGP-361.}

\subjclass[2011]{Primary 47B49 46L40 (46L57 47B47 47D25 47A 15A99 16S50 47L30)} 

\keywords{Local homomorphism, local $^*$-homomorphism; 2-local homomorphism, 2-local $^*$-homomorphism; 2-local automorphism, 2-local $^*$-automorphism}

\date{}
\maketitle

\begin{abstract} It is established that every (not necessarily linear) 2-local $^*$-homomorphism from a von Neumann algebra into a C$^*$-algebra is linear and a $^*$-homomorphism. In the setting of (not necessarily linear) 2-local $^*$-homomorphism from a compact C$^*$-algebra we prove that the same conclusion remains valid. We also prove that every 2-local Jordan $^*$-homomorphism from a JBW$^*$-algebra into a JB$^*$-algebra is linear and a Jordan $^*$-homomorphism.
\end{abstract}

\maketitle
\thispagestyle{empty}

\section{Introduction}\label{sec:intro}

The Gleason-Kahane-\.{Z}elazko theorem (cf. \cite{Gle,KaZe}), a fundamental contribution in the theory of Banach algebras, asserts that every unital linear functional $F$ on a complex unital Banach algebra $A$ such that, $F (a)$ belongs to the spectrum, $\sigma (a)$, of $a$ for every $a\in A$, is multiplicative. In modern terminology, this is equivalent to say that every unital linear local homomorphism from a unital complex Banach algebra $A$ into $\mathbb{C}$ is multiplicative. We recall that a linear mapping $T$ from a Banach algebra $A$ into a Banach algebra $B$ is said to be a \emph{local homomorphism} if for every $a$ in $A$ there exists a homomorphism  $\Phi_{a} : A \to B$, depending on $a$, satisfying $T(a) = \Phi_a (a)$. \emph{Local derivations} are similarly defined. Kadison \cite{Kad90} and Larson and Sourour \cite{LarSou} made the first contributions to the theory of local derivations and local automorphisms on Banach algebras, respectively. Briefly speaking, Johnson  \cite{John01}, culminated the studies on local derivations, showing that every local derivation from a C$^*$-algebra $A$ into a Banach $A$-bimodule is a derivation. A wide list of authors, studied local homomorphisms between C$^*$-algebras (we refer to the introduction of \cite{Pe2014} for a recent expository paper on linear local homomorphisms).\smallskip

After the Gleason-Kahane-\.{Z}elazko theorem was established, Kowalski and S{\l}odkowski \cite{KoSlod} showed that at the cost of requiring the local behavior at two points, the condition of linearity can be dropped, that is, suppose $A$ is a complex Banach algebra (not necessarily commutative nor unital), then every (not necessarily linear) mapping $T: A \to \mathbb{C}$ satisfying $T(0)=0$ and $T(x-y)\in \sigma (x-y),$ for every $x,y\in A$, is multiplicative and linear.\smallskip

Let $A$ and $B$ be two (complex) Banach algebras. Following the standard notation introduced by P. \v{S}emrl in \cite{Semrl97}, a (not necessarily linear nor continuous) mapping $T: A\to B$ is said to be a \emph{2-local homomorphism} (respectively, a \emph{2-local isomorphism}) if for every $a,b\in A$ there exists a bounded (linear) homomorphism (respectively, a bounded isomorphism) $\Phi_{a,b}: A\to B$, depending on $a$ and $b$, such that $\Phi_{a,b} (a) = T(a)$ and $\Phi_{a,b}(b) = T(b)$. \emph{2-local Jordan homomorphisms}, \emph{2-local Jordan monomorphisms} and \emph{2-local Jordan automorphisms} are defined in a similar fashion. We recall that a linear mapping $\Phi: A\to B$ is said to be a Jordan homomorphism whenever $\Phi (a^2) = \Phi (a)^2$ (equivalently, $\Phi$ preserves the Jordan products of the form $a\circ b := \frac12 (a b + ba)$).\smallskip

When $A$ and $B$ are C$^*$-algebras, a mapping $T: A\to B$ is called a \emph{2-local $^*$-homomorphism} (respectively, a \emph{2-local $^*$-isomorphism}) if for every $a,b\in A$ there exists a {$^*$-homomorphism} (respectively, a $^*$-isomorphism) $\Phi_{a,b}: A\to B$, depending on $a$ and $b$, such that $\Phi_{a,b} (a) = T(a)$ and $\Phi_{a,b}(b) = T(b)$.  In the case $A=B$, 2-local isomorphisms, 2-local Jordan isomorphisms, 2-local $^*$-isomorphisms, and 2-local Jordan $^*$-isomorphisms are called \emph{2-local automorphisms}, \emph{2-local Jordan automorphisms}, \emph{2-local $^*$-automorphisms}, and \emph{2-local Jordan $^*$-automorphisms}, respectively.\smallskip

According to this more recent notation, the result established by Kowalski and S{\l}odkowski in \cite{KoSlod} proves that every (not necessarily linear) 2-local homomorphism $T$ from a  (not necessarily commutative nor unital) complex Banach algebra $A$ into the complex field $\mathbb{C}$ is linear and multiplicative. Consequently, every (not necessarily linear) 2-local homomorphism $T$ from $A$ into a commutative C$^*$-algebra is linear and multiplicative.\smallskip

In 1997, \v{S}emrl \cite{Semrl97} proves that for every infinite-dimensional separable Hilbert space $H$, every 2-local automorphism $T: B(H) \to B(H)$ is an automorphism. Short and elegant proofs of \v{S}emrl's theorem for 2-local automorphisms of matrix algebras were obtained by Molnár \cite{Mol2003} and Kim and Kim \cite{KimKim04}. In the just quoted paper \cite{Mol2003}, Molnár also proves that 2-local automorphisms of operator algebras containing all compact operators on a Banach space with a Schauder basis are automorphisms. More studies on 2-local Jordan automorphisms on the algebra of all $n\times n$ real or complex matrices were developed by Fo\v{s}ner in \cite{Fos2012}, where it is additionally established that every 2-local Jordan automorphism of any subalgebra of $B(X)$ which contains the ideal of all compact operators on $X$, where $X$ is a real or complex separable Banach space is either an automorphism or an anti-automorphism. The same author investigates, in \cite{Fos2014}, 2-local $^*$-automorphisms, 2-local $^*$-antiautomorphisms, and 2-local Jordan $^*$-derivations on $M_{n} (\mathbb{C})$, $B(H)$ and certain unital standard operator algebras on $H$ with involution.
\smallskip

In 2012, Ayupov and Kudaybergenov introduce new techniques to generalize \v{S}emrl's theorem for arbitrary Hilbert spaces, showing that 2-local automorphisms on the algebra $B(H)$ on an arbitrary (no separability is assumed) Hilbert space $H$ are automorphisms (see \cite{AyuKuday2012}).\smallskip

Assuming linearity Hadwin and Li \cite[Theorem 3.7]{HadLi04} prove that every bounded linear and unital 2-local homomorphism (respectively, 2-local $^*$-homomorphism) from a unital C$^*$-algebra of real rank zero into itself is a homomorphism (respectively, a $^*$-homomorphism). Under the same additional assumptions, Pop \cite{Pop} establishes that every bounded linear 2-local homomorphism (respectively, 2-local $^*$-homomorphism) from a von Neumann algebra into another C$^*$-algebra is a homomorphism (respectively, a $^*$-homomorphism) \cite[Corollary 3.6]{Pop}. In 2006, Liu and Wong prove that every linear 2-local automorphism $T$ of a C$^*$-algebra whose range is a C$^*$-algebra is an algebra homomorphism. Actually every bounded linear 2-local homomorphism between C$^*$-algebras is a homomorphism (cf. \cite{Pe2014}).\smallskip

In the setting of C$^*$-algebras, Kim and Kim \cite{KimKim05} prove that every surjective 2-local $^*$-automorphism on a prime C$^*$-algebra or on a C$^*$-algebra such that the identity element is properly infinite is a $^*$-automorphism. An illustrative example provided by Gy\H{o}ry in \cite{Gy} proves the existence of non-surjective linear 2-local $^*$-automorphisms between $C_0(L)$-spaces, which shows that the hypothesis concerning surjectivity cannot be relaxed in the result established by Kim and Kim. \smallskip

In this paper we study (not necessarily linear nor continuous) 2-local homomorphisms and 2-local $^*$-homomorphisms between general C$^*$-algebras. Our main result (Theorem \ref{t 2-local hom on von Neumann algebras}) establishes that every (not necessarily linear) 2-local $^*$-homomorphism from a von Neumann algebra into a C$^*$-algebra is linear and a $^*$-homomorphism. The techniques involve several applications of the Bunce-Wright-Mackey-Gleason theorem \cite{BuWri92,BuWri94}, and a subtle variant for dual or compact C$^*$-algebras due to Aarnes \cite{Aarnes70}. In the setting of (not necessarily linear) 2-local $^*$-homomorphism from a compact C$^*$-algebra we prove that the conclusion of Theorem \ref{t 2-local hom on von Neumann algebras} remains valid (Theorem \ref{t 2-local compact C*-algebras}).\smallskip

Finally, in section 4, we prove that every (not necessarily linear) 2-local Jordan $^*$-homomorphism from a JBW$^*$-algebra into a JB$^*$-algebra is linear and a Jordan $^*$-homomorphism (compare Theorem \ref{t 2-local Jordan hom on JBW-algebras}).\smallskip

It would be of great interest to extend Theorem \ref{t 2-local hom on von Neumann algebras} (respectively,  Theorem \ref{t 2-local Jordan hom on JBW-algebras}) to C$^*$-algebras (respectively, JB$^*$-algebras), but in this case projections and partial isometries are useless, because there exist C$^*$-algebras lacking of non-trivial projections.

\section[2-local $^*$-homomorphisms]{2-local $^*$-homomorphisms on von Neumann algebras}

When $A$ is a C$^*$-algebra or a JB$^*$-algebra, the symbol $A_{sa}$ will
stand for the set of all self-adjoint elements in $A$.\smallskip

We begin our study gathering some basic properties on 2-local mappings. A mapping $f$ between Banach algebras $A$ and $B$ is said to be \emph{zero products preserving} if the implication $$a b =0 \Rightarrow f(a) f(b) =0$$ holds for every $a,b\in A$. Let us recall that elements $a, b$ in a C$^*$-algebra $A$ are said to be \emph{orthogonal} (denoted by $a \perp b$) whenever $a b^* = b^* a=0$. A map $f$ from $A$ to another C$^*$-algebra is called \emph{orthogonality preserving} when $f(a)\perp f(b),$ for every $a,b\in A$ with $a\perp b$.\smallskip

Throughout this note, given a 2-local homomorphism (respectively, a 2-local $^*$-homomorphism) $T:A\to B$ between Banach algebras (respectively, between C$^*$-algebras) and elements $a,b\in A$, the symbol  $\Phi_{a,b}$ will denote a (linear) homomorphism (respectively, $^*$-homomorphism) satisfying $T(a) = \Phi_{a,b} (a)$ and $T(b) = \Phi_{a,b}(b)$.

\begin{lemma}\label{l basic properties} Let $T: A\to B$ be a (not necessarily linear nor continuous) 2-local homomorphism between (complex) Banach algebras. The following statements hold:\begin{enumerate}[$(a)$] \item $T$ is 1-homogeneous, that is, $T(\lambda a) = \lambda T(a)$ for every $a\in A$, $\lambda\in \mathbb{C}$;
\item $T$ maps idempotents in $A$ to idempotents in $B$;
\item $T$ is zero-products preserving;
\item $T(a)^2 = T(a^2)$, for every $a\in A$.
\end{enumerate}
\noindent Under the additional hypothesis of $T$ being a 2-local $^*$-homomorphism between C$^*$-algebras we have:
\begin{enumerate}[$(e)$]
\item[$(e)$]
 $T$ is 1-Lipschitzian and automatically continuous, that is, $$\| T(a) - T(b)\| \leq \| a-b\|,$$ for every $a,b\in A$;
\item[$(f)$]
 $T(a^*) = T(a)^*$, for every $a\in A$;
\item[$(g)$]
 $T$ maps projections in $A$ to projections in $B$;
\item[$(h)$]
 $T$ preserves orthogonality.
\item[$(i)$] $T(a)T(a)^* =  T(a a^*),$ and  $T(a)^* T(a) = T(a^* a)$, for every $a\in A$.
\end{enumerate}
\end{lemma}

\begin{proof}$(a)$ For each $a\in A$, $\lambda\in \mathbb{C}$, let us consider the homomorphism $\Phi_{a,\lambda a}$. Then $T(\lambda a) = \Phi_{a,\lambda a} (\lambda a) = \lambda \Phi_{a,\lambda a} (a) = \lambda T(a)$.\smallskip

$(b)$ For each idempotent $e$ in $A$, we have $$T(e)^2 = \Phi_{e,e} (e)^2 = \Phi_{e,e} (e^2) = \Phi_{e,e} (e) = T(e).$$

$(c)$ Suppose $a b = 0$ for certain $a,b\in A$. Then $$T(a) T(b) = \Phi_{a,b} (a) \Phi_{a,b} (b) = \Phi_{a,b} (a b) =0.$$

$(d)$ Pick an element $a\in A$ and consider the homomorphism $\Phi_{a,a^2}$. Then $$T(a)^2 = \Phi_{a,a^2} (a)^2 = \Phi_{a,a^2} (a^2) = T(a^2).$$

$(e)$ Having in mind that every  $^*$-homomorphism between C$^*$-algebras is contractive, given $a,b\in A$, the $^*$-homomorphism $\Phi_{a,b}$ can be applied to prove the inequality $$\|T(a)-T(b) \| = \| \Phi_{a,b} (a) - \Phi_{a,b} (b) \| = \| \Phi_{a,b} (a- b) \|  \leq \| a-b\|,$$ which gives the desired statement.\smallskip

$(f)$ For each $a$ in $A$, we have $$T(a)^* = \Phi_{a,a^*} (a)^* = \Phi_{a,a^*} (a^*) = T(a^*).$$

The statements $(g)$ and $(h)$ are clear from the previous ones. Finally, to prove $(i)$, we take a $^*$-homomorphism $\Phi_{a,aa^*}$ satisfying $T(a) = \Phi_{a, aa^*} (a)$ and $T(a a^*) = \Phi_{a, aa^*} (a a^*)$, and we observe that $$T(a)T(a)^* = \Phi_{a, aa^*} (a) \Phi_{a, aa^*} (a)^* = \Phi_{a, aa^*} (a a^*) = T(a a^*).$$
\end{proof}

The next technical lemma establishes that every 2-local homomorphism between Banach algebras is additive on a couple of idempotents whose products are zero.

\begin{lemma}\label{l additivity on orthogonal idempotents} Let $T: A\to B$ be a (not necessarily linear nor continuous) 2-local homomorphism between (complex) Banach algebras. Let $e$ and $f$ be two idempotents in $A$ satisfying $e f = f e =0$. Then $T(e + f) = T(e) + T(f)$.
\end{lemma}

\begin{proof} The identity $T(e)+T(f-e)=\Phi_{e,f-e}(e)+\Phi_{e,f-e}(f-e)=\Phi_{e,f-e}(e+(f-e))=\Phi_{e,f-e}(f),$ implies that $T(e)+T(f-e)$ is an idempotent in $B$. Therefore, by Lemma \ref{l basic properties},
$$T(e)+T(f-e)=(T(e)+T(f-e))^2$$ $$=T(e)^2+T(f-e)^2+T(e)T(f-e)+T(f-e)T(e)$$ $$=T(e)+T((f-e)^2)+\Phi_{e,f-e} (e(f-e))+ \Phi_{e,f-e} ((f-e)e)$$ $$= T(e)+T(f+e)-T(e)-T(e),$$ which gives $$2T(e)=T(e+f)+T(e-f).$$ Replacing $e$ with $f$ we get $$ 2T(f)=T(e+f)+T(f-e) = T(e+f)- T(e-f),$$ and hence
$T(e)+T(f)=T(e+f).$
\end{proof}

We shall establish now the linearity of every 2-local homomorphism on a finite linear combination of idempotents having zero products.

\begin{lemma}\label{l linearity on orthogonal idempotents} Let $T: A\to B$ be a (not necessarily linear nor continuous) 2-local homomorphism between (complex) Banach algebras. Let $e_1,\ldots, e_n$ be idempotents in $A$ satisfying $e_i e_j = e_j e_i =0$ for every $i\neq j$. Then \begin{enumerate}[$(a)$] \item $\displaystyle T\left(\sum_{i=1}^{n} e_i \right) = \sum_{i=1}^{n} T(e_i) $;
\item $\displaystyle  T\left(\sum_{i=1}^{n} \lambda_i e_i \right) = \sum_{i=1}^{n} \lambda_i T(e_i) $, for every $\lambda_1, \ldots, \lambda_n\in \mathbb{C}$.
\end{enumerate}
\end{lemma}

\begin{proof} We shall prove $(a)$ by induction on $n$. The case $n=1$ is clear, while the case $n=2$ is established in Lemma \ref{l additivity on orthogonal idempotents}. Let $e_1,\ldots, e_n, e_{n+1}$ be idempotents in $A$ satisfying $e_i e_j = e_j e_i =0$ for every $i\neq j$. The elements $e = e_1 +\ldots +e_n$ and $e_{n+1}$ are idempotents in $A$ satisfying $ e e_{n+1} = e_{n+1} e=0$. Lemma \ref{l additivity on orthogonal idempotents} combined with the induction hypothesis assure that $$T\left(\sum_{i=1}^{n+1} e_i \right) = T (e + e_{n+1} ) = T(e) + T(e_{n+1}) = \sum_{i=1}^{n} T(e_i)  + T(e_{n+1}).$$

$(b)$ As in the proof of Lemma \ref{l additivity on orthogonal idempotents}, $\Phi_{a,b}$ will denote a (linear) homomorphism satisfying $T(a) = \Phi_{a,b} (a)$ and $T(b) = \Phi_{a,b}(b)$. Fix $j\in \{ 1,\ldots, n\}$ and set $z= \displaystyle \sum_{i=1}^{n} \lambda_i e_i$. We first calculate \begin{equation}\label{eq 1 lemma 2} T\left(\sum_{i=1}^{n} \lambda_i e_i \right) T(e_j) = \Phi_{z,e_j} \left( \sum_{i=1}^{n} \lambda_i e_i \right) \Phi_{z,e_j} \left( e_j \right)
 \end{equation}$$= \Phi_{z,e_j} \left(\left( \sum_{i=1}^{n} \lambda_i e_i \right) e_j \right) = \Phi_{z,e_j} (\lambda_j e_j ) = \lambda_j T(e_j).$$ To simplify notation, we write $e=\displaystyle \sum_{i=1}^{n} e_i$. In this case $$T\left(\sum_{i=1}^{n} \lambda_i e_i \right) = \Phi_{z,e} \left(\sum_{i=1}^{n}   \lambda_i e_i \right)    = \Phi_{z,e} \left(\left(\sum_{i=1}^{n}  \lambda_i e_i \right) e \right) $$ $$= \Phi_{z,e} \left(\sum_{i=1}^{n}  \lambda_i e_i \right) \Phi_{z,e} \left( e \right)= T \left(\sum_{i=1}^{n}  \lambda_i e_i \right) T \left( e \right)= T \left(\sum_{i=1}^{n}  \lambda_i e_i \right)  T \left(\sum_{j=1}^{n} e_j \right)$$ $$ = \hbox{(by $(a)$)} = T \left(\sum_{i=1}^{n}  \lambda_i e_i \right)  \sum_{j=1}^{n} T(e_j) = \hbox{(by \eqref{eq 1 lemma 2})} = \sum_{j=1}^{n} \lambda_j T(e_j).$$
\end{proof}

\begin{lemma}\label{l complex combinations hermitians}
Let $T: A\to B$ be a (not necessarily linear) 2-local $^*$-homomor-phism between C$^*$-algebras. Then $T (a + i b) = T(a) + i T(b),$ for every $a,b\in A_{sa}$.
\end{lemma}

\begin{proof}
Given $a,b\in A_{sa}$, we consider the $^*$-homomorphisms $\Phi_{a,a+ib}$ and $\Phi_{b,a+ib}$. The identities $$T(a+ i b) = \Phi_{a,a+ib} (a) + i \Phi_{a,a+ib} (b)= T(a) +   i \Phi_{a,a+ib} (b),$$ and
$$T(a+i b) = \Phi_{b,a+ib} (a) + i \Phi_{b,a+ib} (b)= \Phi_{b,a+ib} (a) +   i T (b),$$ together with Lemma \ref{l basic properties}$(f)$, imply that $$ T(a+i b) + T(a-i b) = T(a+i b) + T(a+i b)^* = 2 T(a),$$  and $$ 2 i T(b) = T(a+i b) - T(a-i b),$$ which prove $T(a+ ib) = T(a) + i T(b).$
\end{proof}

\begin{theorem}\label{t 2-local hom on von Neumann algebras not containing M_2} Let $\mathcal{M}$ be a von Neumann algebra with no Type $I_2$ direct summand and let $B$ be a C$^*$-algebra. Suppose $T : \mathcal{M} \to B$ is a (not necessarily linear) 2-local $^*$-homomorphism. Then $T$ is linear and a $^*$-homomorphism.
\end{theorem}

\begin{proof} By Lemma \ref{l basic properties}$(e)$,  $T$ is 1-Lipschitzian and hence automatically continuous. Let $\mathcal{P} (\mathcal{M})$ denote the lattice of projections of $\mathcal{M}$ and define a mapping $\mu : \mathcal{P} (\mathcal{M}) \to B$ by $\mu (p ) = T(p),$ for every $p\in \mathcal{P} (\mathcal{M})$. Lemma \ref{l additivity on orthogonal idempotents} implies that $\mu (p + q ) = \mu (p) +\mu (q)$ whenever $p q = 0$ in $\mathcal{P} (\mathcal{M})$, that is, $\mu$ is finitely additive. Furthermore, by Lemma \ref{l basic properties}$(g)$, $T(p)$ is a projection in $B$ for every $p\in \mathcal{P} (\mathcal{M})$. Therefore, $$\|\mu (p) \| = \|T(p)\| \leq 1,$$ for every $p\in \mathcal{P} (\mathcal{M})$.\smallskip

Therefore the above mapping $\mu$ is a bounded $B$-valued finitely additive measure on $\mathcal{P} (\mathcal{M})$. By the Bunce-Wright-Mackey-Gleason theorem (cf. \cite[Theorem A]{BuWri92} or \cite{BuWri94}), there exists a bounded linear operator $G : \mathcal{M} \to B$ satisfying \begin{equation}\label{eq Mackey-Gleaason} T(p) = \mu (p) = G(p), \end{equation}for every $p\in \mathcal{P} (\mathcal{M})$. Let us consider an algebraic element $\displaystyle b = \sum_{i=1}^{n} \lambda_i p_i$ in $\mathcal{M}$, where $\lambda_i\in \mathcal{C}$ and $p_1,\ldots, p_n$ are mutually orthogonal projections in $\mathcal{M}$. Lemma \ref{l linearity on orthogonal idempotents}$(b)$ and \eqref{eq Mackey-Gleaason} assure that $$T(b) = \sum_{i=1}^n \lambda_i T(p_i) = \sum_{i=1}^n \lambda_i G(p_i) = G(b),$$ for every algebraic element $b$ in $\mathcal{M}$.\smallskip

Since every self-adjoint element in a von Neumann algebra can be approximated in norm by algebraic elements, it follows from the continuity of $T$ and $G$ that $T(a ) = G (a)$ for every $a= a^*$ in $\mathcal{M}$ and consequently \begin{equation}\label{eq linearity on hermitian} T(a+b) = G(a+b) = G(a) + G(b) = T(a) +T(b)
\end{equation} for every $a,b\in \mathcal{M}_{sa}$, that is, $T|_{\mathcal{M}_{sa}} : \mathcal{M}_{sa} \to B$ is linear. Lemma \ref{l complex combinations hermitians} implies that $T$ is linear. Finally \cite[Corollary 3.6]{Pop} or \cite[Theorem 3.9]{Pe2014}, imply that the mapping $T$ is a $^*$-homomorphism.
\end{proof}

The case of von Neumann algebras containing a Type $I_2$ direct summand must be treated independently.

\begin{lemma}\label{l 2-local on factors} Let $\mathcal{M}$ be a von Neumann algebra factor, $B$ a C$^*$-algebra, and let $T: \mathcal{M} \to B$ be a (not necessarily linear) 2-local $^*$-homomorphism. Then the following statements hold:\begin{enumerate}[$(a)$] \item If there exists $a\in \mathcal{M}$ with $a \neq 0$ and $T(a)=0$ then $T=0$;
\item If $T\neq 0$ then $T$ is a 2-local $^*$-monomorphism, that is, for every $a,b\in \mathcal{M}$ there exists a $^*$-monomorphism $\Phi_{a,b}$ satisfying $T(a) = \Phi_{a,b} (a)$ and $T(b) = \Phi_{a,b} (b)$;
\item If $T\neq 0$ then $T$ is an isometry, that is, $\|T(a)\| = \| a\|$, for every $a\in \mathcal{M}$.
\end{enumerate}
\end{lemma}

\begin{proof} Since the kernel of every $^*$-homomorphism $\pi: \mathcal{M }\to B$ is a weak$^*$-closed ideal of $\mathcal{M}$, we can easily see that a every non-zero $^*$-homomorphism $\pi: \mathcal{M }\to B$ is a $^*$-monomorphism and an isometry.\smallskip

$(a)$ Suppose $T(a)=0$, for an element $a\in \mathcal{M}\backslash \{0\}$. For each $b$ in $\mathcal{M}$, take a $^*$-homomorphism $\Phi_{a,b}$ satisfying  $0=T(a) = \Phi_{a,b} (a)$ and $T(b) = \Phi_{a,b} (b)$. It follows from the above that $\Phi_{a,b}=0$ and hence $T(b)=0$. The statements $(b)$ and $(c)$ are clear from the above.
\end{proof}

\begin{proposition}\label{p 2-local *-hom on M2} Let $B$ be a C$^*$-algebra and let $T: M_2 (\mathbb{C}) \to B$ be a (not necessarily linear) 2-local $^*$-homomorphism. Then $T$ is linear and a $^*$-homomorphism.
\end{proposition}

\begin{proof} To simplify notation, let us write $e_1= \left(
                                                         \begin{array}{cc}
                                                           0 & 1 \\
                                                           0 & 0 \\
                                                         \end{array}
                                                       \right)$, $ e_2= e_1^*= \left(
                                                                          \begin{array}{cc}
                                                                            0 & 0 \\
                                                                            1 & 0 \\
                                                                          \end{array}
                                                                        \right)$, $p_1 = \left(
                                                                                           \begin{array}{cc}
                                                                                             1 & 0 \\
                                                                                             0 & 0 \\
                                                                                           \end{array}
                                                                                         \right)$, and $p_2= \left(
                                                                                                           \begin{array}{cc}
                                                                                                             0 & 0 \\
                                                                                                             0 & 1 \\
                                                                                                           \end{array}
                                                                                                         \right)
                                                                                         $. 
By Lemma \ref{l basic properties}, $T(a^2) = T(a)^2$, $T(a^*) = T(a)^*$ and $T(a)T(a)^* =  T(a a^*),$ and $T(a)^* T(a) = T(a^* a)$, for every $a\in M_2 (\mathbb{C})$. To simplify notation we set $z= \lambda p_1 + \mu e_1 + \alpha e_2 + \beta p_2$.  Considering the $^*$-homomorphisms $\Phi_{e_1,z}$ we have
$$ \lambda T(p_1) + \mu T(e_1) + \alpha T(e_2) + \beta T(p_2) = \lambda T(e_1 e_1^*) + \mu T(e_1) + \alpha T(e_1^*) + \beta T(e_1^* e_1) $$
$$= \lambda T(e_1) T(e_1)^* + \mu T(e_1) + \alpha T(e_1)^* + \beta T(e_1)^* T(e_1)$$
$$= \lambda \Phi_{e_1,z} (e_1) \Phi_{e_1,z} (e_1)^* + \mu \Phi_{e_1,z} (e_1) + \alpha \Phi_{e_1,z} (e_1)^* + \beta \Phi_{e_1,z} (e_1)^* \Phi_{e_1,z} (e_1)$$
$$= \Phi_{e_1,z} ( \lambda e_1 e_1^* + \mu e_1 + \alpha e_1^* + \beta e_1^* e_1)  $$
$$= \Phi_{e_1,z} ( \lambda p_1 + \mu e_1 + \alpha e_2 + \beta p_2 ) = T(\lambda p_1 + \mu e_1 + \alpha e_2 + \beta p_2).$$ Since $\{ p_1, e_1 , e_2, p_2\}$ is a basis of $M_2 (\mathbb{C})$, the above identity shows that $T$ is linear. The proof concludes by \cite[Theorem 3.9]{Pe2014}.
\end{proof}

We shall establish now a strengthened version of Lemma \ref{l linearity on orthogonal idempotents} for 2-local $^*$-homomorphisms. We recall that an element $e$ in a C$^*$-algebra $A$ is said to be a \emph{partial isometry} when $ee^*$ (equivalently, $e^*e$) is a projection. For each partial isometry $e\in A$, the elements $ee^*$ and $e^*e$ are called the left and right support projections of $e$, respectively.

\begin{lemma}\label{l linearity on orthogonal tripotents} Let $T: A\to B$ be a (not necessarily linear) 2-local $^*$-homomor-phism between C$^*$-algebras. Let $e_1,\ldots, e_n$ be a family of mutually orthogonal partial isometries in $A$. Then $$\displaystyle  T\left(\sum_{i=1}^{n} \lambda_i e_i \right) = \sum_{i=1}^{n} \lambda_i T(e_i), $$ for every $\lambda_1, \ldots, \lambda_n\in \mathbb{C}$.
\end{lemma}

\begin{proof} Let us note that by Lemma \ref{l basic properties}$(i)$, and since $\displaystyle \sum_{i=1}^{n} e_i$ is a partial isometry, $T(e_i)$, and  $\displaystyle T\left(\sum_{i=1}^{n} e_i\right)$ are partial isometries in $B$. Lemma \ref{l basic properties}$(h)$ assures that $T(e_i) \perp T(e_j)$, for every $i\neq j$, and hence $\displaystyle \sum_{i=1}^{n} T\left(e_i\right)$ also is a partial isometry in $B$ (cf. Lemma \ref{l basic properties}$(i)$).\smallskip

Take $\lambda_1, \ldots, \lambda_n\in \mathbb{C}$ and set $z= \sum_{i=1}^{n} \lambda_i e_i.$ The identity $$ T\left(\sum_{i=1}^{n} \lambda_i e_i \right) T(e_j^* e_j) = T\left(\sum_{i=1}^{n} \lambda_i e_i \right) T(e_j) ^* T(e_j) $$ $$=\Phi_{z,e_j} \left(\sum_{i=1}^{n} \lambda_i e_i \right) \Phi_{z,e_j} (e_j) ^* \Phi_{z,e_j}(e_j) = \lambda_j \Phi_{z,e_j} (e_j e_j^* e_j) $$ $$= \lambda_j \Phi_{z,e_j} (e_j) = \lambda_j T (e_j) , $$ is valid for every $j.$ Since $e_1^* e_1,\ldots , e_n^* e_n$ are mutually orthogonal projections in $A$, Lemma \ref{l additivity on orthogonal idempotents} implies that $$ T\left(\sum_{j=1}^{n} e_j^* e_j \right)  = \sum_{i=1}^{n} T\left( e_j^* e_j \right) .$$ Set $\displaystyle p= \sum_{j=1}^{n} e_j^* e_j $ and consider the $^*$-homomorphism $\Phi_{z,p}$. It follows from the above that $$T\left(\sum_{i=1}^{n} \lambda_i e_i \right)   = T\left( \left(\sum_{i=1}^{n} \lambda_i e_i \right) \left( \sum_{j=1}^{n} e_j^* e_j \right) \right)  $$ $$= \Phi_{z,p} \left( \left(\sum_{i=1}^{n} \lambda_i e_i \right) \left( \sum_{j=1}^{n} e_j^* e_j \right) \right)= \Phi_{z,p} \left( \sum_{i=1}^{n} \lambda_i e_i \right) \Phi_{z,p} \left( \sum_{j=1}^{n} e_j^* e_j \right) $$
$$= T \left( \sum_{i=1}^{n} \lambda_i e_i \right) T \left( \sum_{j=1}^{n} e_j^* e_j \right) = T \left( \sum_{i=1}^{n} \lambda_i e_i \right)  \left( \sum_{j=1}^{n} T(e_j^* e_j) \right) $$ $$= \sum_{j=1}^{n}  T \left( \sum_{i=1}^{n} \lambda_i e_i \right) T(e_j^* e_j) = \sum_{j=1}^{n} \lambda_j T(e_j).$$\end{proof}

Let $a$ be an element in a von Neumann algebra $\mathcal{M}$. Following the notation in \cite[\S 1.10]{Sak}, the least projection $p$ in $\mathcal{M}$ such that $a p = a$ (respectively, $p a = a$) is called the \emph{right support projection} (respectively, the \emph{left support projection}) of $a$ and is denoted by $r(a)$ (respectively, $l(a)$). If $a$ is self-adjoint, $l(a)= r(a)$ is simply called the \emph{support} projection of $a$ and is denoted by $s(a)$.\smallskip

We recall at this point that a mapping $f$ from a C$^*$-algebra $A$ into a Banach space $B$ is said to be \emph{orthogonally additive} if for every $a,b$ in $A$ with $a \perp b$, we have $f(a+b) = f(a) +f(b)$.\smallskip

\begin{proposition}\label{p 2-local is OA} Let $\mathcal{M}$ be a von Neumann algebra, let $B$ be a C$^*$-algebra, and Let $T : \mathcal{M} \to B$ be a (not necessarily linear) 2-local $^*$-homomorphism. Then $T$ is orthogonally additive.
\end{proposition}

\begin{proof} Let $a$ and $b$ be two orthogonal elements in $\mathcal{M}$. We consider the polar decompositions of $a$ and $b$ in the form $a = u_a |a|$ and $b =u_b |b| $, where $|a| = (a^* a)^{\frac12}$, $|b| = (b^* b)^{\frac12}$, $u_a$ and $u_b$ are partial isometries in $\mathcal{M}$ with $u_a^* u_a = s(|a|)$ and $u_b^* u_b = s(|b|)$ (cf. \cite[\S 1.12]{Sak}). Since $a\perp b$ we have $|a| \perp |b|$, $u_a \perp u_b$ and $| a+b| = |a|+ |b|$.\smallskip

Let $\mathcal{M}_{|a|}$, $\mathcal{M}_{|b|}$ and $\mathcal{M}_{\{|a|,|b|\}}$ denote the von Neumann subalgebras of $\mathcal{M}$ generated by $|a|, |b|$ and $\{|a|,|b|\}$, respectively. Since $\mathcal{M}_{|a|}$ and $\mathcal{M}_{|b|}$ are abelian von Neumann algebras and $\mathcal{M}_{\{|a|,|b|\}}$ coincides with the orthogonal sum $ \mathcal{M}_{|a|} \oplus^{\infty} \mathcal{M}_{|b|}$, we deduce that $\mathcal{M}_{\{|a|,|b|\}}$ is an abelian von Neumann algebra. We define a bounded linear operator $\Psi : \mathcal{M}_{|a|} \oplus^{\infty} \mathcal{M}_{|b|} \to \mathcal{M}$ given by $$\Psi (x+y) := u_a x + u_b y, \ (x\in \mathcal{M}_{|a|}, y \in \mathcal{M}_{|b|}).$$

We consider the mapping $T\circ \Psi : \mathcal{M}_{|a|} \oplus^{\infty} \mathcal{M}_{|b|} \to B$. We observe that, for each projection $p$ in $\mathcal{M}_{|a|} \oplus^{\infty} \mathcal{M}_{|b|}$, $\Psi (p)$ is a partial isometry in $M$, and hence $T\circ\Psi (p)$ also is a partial isometry in $B$ (compare Lemma \ref{l basic properties}$(i)$ or the first part in the proof of Lemma \ref{l linearity on orthogonal tripotents}), which gives $\| T\circ\Psi (p) \|\leq 1$, for every projection $p$ in $\mathcal{M}_{|a|} \oplus^{\infty} \mathcal{M}_{|b|}$.\smallskip

Let $p_1,\ldots, p_n$ be mutually orthogonal projections in $\mathcal{M}_{|a|} \oplus^{\infty} \mathcal{M}_{|b|}$. Since $\Psi(p_1),$ $\ldots,$ $ \Psi(p_n)$ are mutually orthogonal partial isometries in $\mathcal{M}$, Lemma \ref{l linearity on orthogonal tripotents} proves that $$T\circ \Psi\left(\sum_{i=1}^{n} p_i \right) = T\left(\sum_{i=1}^{n} \Psi(p_i) \right) =\sum_{i=1}^{n}  T\left(\Psi(p_i) \right).$$ That is, $T\circ \Psi : \mathcal{P}\left(\mathcal{M}_{|a|} \oplus^{\infty} \mathcal{M}_{|b|}\right) \to B$, $p \mapsto T\Psi (p)$, is a bounded $B$-valued finitely additive measure on the set $\mathcal{P} (\mathcal{M}_{|a|} \oplus^{\infty} \mathcal{M}_{|b|})$ of all projections in $\mathcal{M}_{|a|} \oplus^{\infty} \mathcal{M}_{|b|}$. Since $\mathcal{M}_{|a|} \oplus^{\infty} \mathcal{M}_{|b|}$ is an abelian von Neumann algebra, it follows from the Bunce-Wright-Mackey-Gleason theorem (cf. \cite[Theorem A]{BuWri92} or \cite{BuWri94}), that there exists a bounded linear operator $G : \mathcal{M}_{|a|} \oplus^{\infty} \mathcal{M}_{|b|} \to B$ satisfying $$ T\circ \Psi (p) = G(p), $$ for every $p\in \mathcal{P} (\mathcal{M}_{|a|} \oplus^{\infty} \mathcal{M}_{|b|})$. The continuity argument applied in the proof of Theorem \ref{t 2-local hom on von Neumann algebras not containing M_2} shows that $$T\circ \Psi (z) = G(z),$$ for every $z= z^*$ in $\mathcal{M}_{|a|} \oplus^{\infty} \mathcal{M}_{|b|}$, and then $$T\circ \Psi (z_1+z_2) = G(z_1+ z_2) = G(z_1) + G(z_2) = T\circ \Psi (z_1)+ T\circ \Psi (z_2),$$ for every $z_1,z_1 \in \left( \mathcal{M}_{|a|} \oplus^{\infty} \mathcal{M}_{|b|} \right)_{sa}$. Taking $z_1 = |a|$ and $z_2 = |b|$ we obtain $T(a+b) = T(a)+ T(b).$
\end{proof}

By a simple induction argument, combined with Proposition \ref{p 2-local is OA}, we get:

\begin{corollary}\label{c stability for ell infty sums} Let $\left(\mathcal{M}_i\right)_{i=1}^{n}$ be a finite family of von Neumann algebras and let $B$ be a C$^*$-algebra. Suppose that, for every $i$, every 2-local $^*$-homomorphism $T: \mathcal{M}_i \to B$ is linear. Then every 2-local $^*$-homomorphism $\displaystyle T: \bigoplus_{i=1,\ldots,n}^{\ell_{\infty}} \mathcal{M}_i \to B$ is linear. $\hfill \Box$
\end{corollary}

\begin{corollary}\label{c type I2} Every (not necessarily linear) 2-local $^*$-homomorphism from a Type $I_2$ von Neumann algebra into a C$^*$-algebra is linear and a $^*$-homomorphism.
\end{corollary}

\begin{proof} Let $\mathcal{M}$ be a Type $I_2$ von Neumann algebra and let $T : \mathcal{M} \to B$ be a 2-local $^*$-homomorphism from $\mathcal{M}$ into a C$^*$-algebra. By standard classification theory of von Neumann algebras (see, for example, \cite[Theorem 2.3.3]{Sak}) we may suppose that $$ \mathcal{M} = C(K) \otimes M_2(\mathbb{C}),$$ where $C(K)$ is the algebra of all continuous functions on a compact Stonean space $K$.\smallskip

Let $p_1,\ldots, p_m$ be mutually orthogonal projections in $C(K)$ with $p_1+ \ldots +p_m=1$. The von Neumann subalgebra $$\mathcal{M}_{p_1,\ldots, p_m}=p_1\otimes M_2 (\mathbb{C}) \oplus \ldots \oplus p_m\otimes M_2 (\mathbb{C})$$ is C$^*$-isomorphic to the $\ell_{\infty}$-sum $\displaystyle \bigoplus_{i=1,\ldots, m}^{\ell_{\infty}} M_2(\mathbb{C})$. Since the restricted mapping $T|_{\mathcal{M}_{p_1,\ldots, p_m}}: \mathcal{M}_{p_1,\ldots, p_m} \to B$ is a 2-local $^*$-homomorphism, we deduce, via Proposition \ref{p 2-local *-hom on M2} and Corollary \ref{c stability for ell infty sums}, that $T|_{\mathcal{M}_{p_1,\ldots, p_m}}$ is linear. Fix  $x,y\in \mathcal{M}$. By standard arguments (compare \cite[Lemma 8.3]{Mae}), for each $\varepsilon >0$, there exist a subalgebra of the form $\mathcal{M}_{p_1,\ldots, p_m}$, $x_\varepsilon, y_{\varepsilon}\in \mathcal{M}_{p_1,\ldots, p_m}$ such that $\|x-x_{\varepsilon}\| < \frac{\varepsilon}{4}$, and $\|y-y_{\varepsilon}\| < \frac{\varepsilon}{4}$. Then, by Lemma \ref{l basic properties}$(e)$, $$\|T(x+y)-T(x)-T(y)\| \leq \|T(x+y) -T(x_{\varepsilon}+y_{\varepsilon})\| + \|T(x_{\varepsilon})- T(x) \|$$ $$ +\|T(y_{\varepsilon})-T(y) \| < \| (x+y) -(x_{\varepsilon}+y_{\varepsilon})\| + \|x_{\varepsilon}- x \| +\|y_{\varepsilon}-y \|< \varepsilon. $$ Since $\varepsilon$ was arbitrarily chosen, we get $T(x+y ) = T(x) + T(y).$\end{proof}

The main result of this section is a consequence of Theorem \ref{t 2-local hom on von Neumann algebras not containing M_2}, Corollary \ref{c type I2} and Corollary \ref{c stability for ell infty sums}.

\begin{theorem}\label{t 2-local hom on von Neumann algebras} Every (not necessarily linear) 2-local $^*$-homomorphism from a von Neumann algebra into a C$^*$-algebra is linear and a $^*$-homomorphism.
\end{theorem}

\begin{proof} Let $\mathcal{M}$ be a von Neumann algebra and let $T$ be a 2-local $^*$-homomor-phism from $\mathcal{M}$ into a C$^*$-algebra $B$. By \cite[Proposition 2.2.10 and Theorem 2.3.2]{Sak}, $\mathcal{M}$ decomposes as the $\ell_{\infty}$-sum of two von Neumann algebras $\mathcal{M}_1$ and $\mathcal{M}_2$, where $\mathcal{M}_1$ contains no Type $I_2$ direct summand and $\mathcal{M}_2$ is a Type $I_2$ von Neumann algebra. Theorem \ref{t 2-local hom on von Neumann algebras not containing M_2} and Corollary \ref{c type I2}, $T|_{\mathcal{M}_1}$ and $T|_{\mathcal{M}_2}$ are linear. The linearity of $T$ follows from Corollary \ref{c stability for ell infty sums}.
\end{proof}

We can rediscover now some of the results commented at the introduction.

\begin{corollary}\label{c 2-local *auto von Neumann} Every (not necessarily linear) surjective 2-local $^*$-automor-phism on a von Neumann algebra is a $^*$-automorphism.$\hfill\Box$
\end{corollary}

\section{2-local $^*$-homomorphisms on dual C$^*$-algebras}

A projection $p$ in a C$^*$-algebra $A$ is said to be
\emph{minimal} if $pAp = \mathbb{C} p.$ A partial isometry $e$ in
$A$ is said to be minimal if $e e^*$ (equivalently, $e^* e$) is a
minimal projection. The \emph{socle} of $A$, soc$(A)$, is defined
as the linear span of all minimal projections in $A$. The
\emph{ideal of compact elements} in $A$, $K(A)$, is defined as the norm
closure of soc$(A)$. A C$^*$-algebra is said to be \emph{dual} or
\emph{compact} if $A= K(A)$. 
We refer to \cite[\S 2]{Kap},
\cite{Alex} and \cite{Yli} for the basic references on dual
C$^*$-algebras.\smallskip

Following standard notation, for each hermitian element $h$ in a C$^*$-algebra $A$, the symbol $A_{h}$ will denote the C$^*$-subalgebra of $A$ generated by $h$.

\begin{lemma}\label{l linearity on single generated subalgebras}
Let $T: A\to B$ be a (not necessarily linear) 2-local $^*$-homomor-phism between C$^*$-algebras. Then, for each $h\in A_{sa}$, $T|_{A_{h}} : A_{h} \to B$ is a linear mapping.
\end{lemma}

\begin{proof} 
Consider an element $b\in A_{h}$ of the form $\displaystyle b= \sum_{k=1}^{m} \alpha_k h^k$ and the $^*$-homomorphism $\Phi_{h,b}$. In this case, $$T(b) = \Phi_{h,b} \left(\sum_{k=1}^{m} \alpha_k h^k\right) = \sum_{k=1}^{m} \alpha_k \Phi_{h,b} \left( h\right)^k = \sum_{k=1}^{m} \alpha_k T \left( h\right)^k, $$ which proves that $T$ is linear on the linear span of the set $\{ a^{k} : k\in \mathbb{N}\}$. We conclude by continuity that $T|_{A_{h}}$ is linear.
\end{proof}

Let $A$ be a C$^*$-algebra. Following the notation in \cite{Aarnes70}, a \emph{positive quasi-linear functional} is a function $\rho: A \rightarrow \mathbb{C}$ such that \begin{enumerate}[$(i)$]\item $\rho_{A_{h}}$ is a positive linear functional for each $h\in A_{sa}$;
\item $\rho (a + ib) = \rho (a) + i \rho (b)$, for every $a,b\in A_{sa}$.
\end{enumerate} When the mapping $\rho$ also satisfies that
$\sup \{ \rho(a) : a \in A, \|a\|\leq 1, a\geq 0 \}= 1$, then $\rho$ is called a \emph{quasi-state} on $A$.\smallskip

Let $T: A\to B$ be a (not necessarily linear) 2-local $^*$-homomorphism between C$^*$-algebras. For each positive functional $\phi\in B^{*}_{+}$, Lemmas \ref{l complex combinations hermitians} and \ref{l linearity on single generated subalgebras} assure that $\phi \circ T: A \to \mathbb{C}$ is a positive multiple of a quasi-state on $A$. When $A$ coincides with the C$^*$-algebra $K(H)$ of all compact operators on a complex Hilbert space $H$ with dim$(H)\geq 3$, Corollary 2 in \cite{Aarnes70} implies that $\phi \circ T$ is linear. Having in mind that for each self-adjoint element $h\in B$ we have $\|h\| = \sup \{ \phi (h) : \phi\geq 0, \|\phi\| =1 \}$ (cf. \cite[Proposition 1.5.4]{Sak}), we deduce that $T$ is linear. Combining these arguments with Proposition \ref{p 2-local *-hom on M2} we get:

\begin{proposition}\label{p 2-local compact operators} Let $H$ be a complex Hilbert space, $B$ a C$^*$-algebra, and let $T: K(H)\to B$ be a (not necessarily linear) 2-local $^*$-homomorphism. Then $T$ is linear and a $^*$-homomorphism. $\hfill\Box$
\end{proposition}

By \cite[Theorem 8.2.]{Alex}, we know that every compact or dual C$^*$-algebra $A$ decomposes as a $c_0$-sum of the form $\displaystyle A= \left(\bigoplus_{\lambda} K(H_{\lambda})\right)_{c_0}$, where each $H_{\lambda}$ is a complex Hilbert space. Suppose $B$ is a C$^*$-algebra and $$T: \displaystyle \left(\bigoplus_{\lambda} K(H_{\lambda})\right)_{c_0} \to B$$ is a (not necessarily linear) 2-local $^*$-homomorphism. Proposition \ref{p 2-local compact operators} implies that $T|_{K(H_{\lambda})}: K(H_{\lambda}) \to B$ is a linear homomorphism for every $\lambda$. For each pair of elements $a,b\in A$ and $\varepsilon>0$ we can find a natural $m$, $\lambda_1,\ldots,\lambda_m$, finite dimensional subspaces $\widetilde{H}_{\lambda_i} \subseteq H_{\lambda_i}$ and elements $a_{\varepsilon}, b_{\varepsilon} \in \displaystyle \bigoplus_{i=1,\ldots,m}^{\ell_{\infty}} K(\widetilde{H}_{\lambda_i})$ satisfying $\|a-a_{\varepsilon}\| <\frac{\varepsilon}{4}$, and $\|b-b_{\varepsilon}\| <\frac{\varepsilon}{4}$. Set $A_1 := \displaystyle \bigoplus_{i=1,\ldots,m}^{\ell_{\infty}} K(\widetilde{H}_{\lambda_i})$. By Corollary \ref{c stability for ell infty sums}, the mapping $\displaystyle T|_{A_1}: A_1 \to B$ is linear. Therefore, by Lemma \ref{l basic properties}$(e)$, $$\left\|T(a+b) -T(a)-T(b) \right\| \leq \left\|T(a+b) -T(a_{\varepsilon}+ b_{\varepsilon}) \right\| + \left\|T(a_{\varepsilon}) -T(a) \right\|$$ $$ + \left\|T(b_{\varepsilon}) -T(b) \right\| \leq \left\|(a+b) -(a_{\varepsilon}+ b_{\varepsilon}) \right\| + \left\|a_{\varepsilon}-a \right\| + \left\|b_{\varepsilon} -b \right\| <\varepsilon.$$ The arbitrariness of $\varepsilon>0$ proves the additivity of $T$. The final statement follows from \cite[Theorem 3.9]{Pe2014}.

\begin{theorem}\label{t 2-local compact C*-algebras} Let $A$ be a dual or compact C$^*$-algebra, $B$ a C$^*$-algebra, and let $T: A\to B$ be a (not necessarily linear) 2-local $^*$-homomorphism. Then $T$ is linear and a $^*$-homomorphism. $\hfill\Box$
\end{theorem}

\begin{corollary}\label{c 2-local *auto prime} Every (not necessarily linear) surjective 2-local $^*$-homomor-phism $T$ on a prime C$^*$-algebra $A,$ with $T(K(A))\neq \{0\}$ is a $^*$-homomorphism. Consequently, every (not necessarily linear) surjective 2-local $^*$-automor-phism on a prime C$^*$-algebra with non-zero socle is a $^*$-automorphism.
\end{corollary}

\begin{proof} By Theorem \ref{t 2-local compact C*-algebras}, $T|_{K(A)} : K(A)\to B$ is linear and a $^*$-homomorphism. Therefore, $\ker (T|_{K(A)})$ is a closed ideal of $A$. We recall that every prime C$^*$-algebra with non-zero socle is primitive and hence its socle is contained in every non-zero closed ideal of $A$. Thus,  $\ker (T|_{K(A)})=\{0\}.$\smallskip

Let us take $a,b\in A$, $x\in K(A).$ By the hypothesis of 2-locality and Theorem \ref{t 2-local compact C*-algebras}, $$T(x) T(a+b) T(x)=T(x(a+b)x) = T(xax) + T(xbx) $$ $$= T(x) T(a) T(x) + T(x) T(b) T(x).$$
Since $T$ is surjective, we can find $z\in A$ such that $T(z) = T(a+b)-T(a)-T(b).$ It follows from the above that $T(xzx)=0$ and hence $xzx=0$, for every $x\in K(A)$. The essentiality of $K(A)$ proves that $z=0$, witnessing that $T(a+b) = T(a)+T(b).$
\end{proof}

\begin{remark}\label{r KimKim} In \cite[Theorem 6]{KimKim05}, the authors prove that for every unital prime C$^*$-algebra $A$, which has a nontrivial idempotent or whose unit element is properly infinite, every surjective 2-local $^*$-automorphism, $T: A\to A$, is a $^*$-automorphism. It should be remarked here that the result is true for every C$^*$-algebra $A$ without requiring any additional hypothesis. Indeed, as noted in \cite{KimKim05}, $T$ is a surjective isometry, which, by the Mazur-Ulam theorem, implies that $T$ is linear. In particular, $T$ is a linear local $^*$-homomorphism, and hence a $^*$-homomorphism by \cite[Theorem 3.9]{Pe2014}.
\end{remark}

\section{2-local Jordan $^*$-homomorphisms on JBW$^*$-algebras}

A \emph{JB$^*$-algebra} or a \emph{Jordan C$^*$-algebra} is a complex Jordan Banach algebra $\mathcal{J}$ equipped with an algebra involution $^*$ satisfying  $\|U_a (a^*) \|= \|a\|^3$, $a\in \mathcal{J}$, where $U_a (x) := 2 (a\circ x) \circ a - a^2 \circ x$. The self-adjoint part $\mathcal{J}_{sa}$ of a JB$^*$-algebra $\mathcal{J}$ is a JB-algebra in the usual sense employed in \cite{Chu2012,HOS}. The reciprocal statement also holds, each JB algebra is the self-adjoint part of a unique JB$^*$-algebra \cite{Wri77}. A JBW$^*$-algebra is a JB$^*$-algebra which is also a dual Banach space. For the standard definitions and properties of JB$^*$- and JBW$^*$-algebras we refer to \cite{HOS,Chu2012} and \cite{Wri77}.\smallskip

Let $\mathcal{J}_1$, $\mathcal{J}_2$ be Jordan Banach algebras. A (not necessarily linear nor continuous) mapping $T: \mathcal{J}_1\to \mathcal{J}_2$ is said to be a \emph{2-local Jordan homomorphism} if for every $a,b\in \mathcal{J}_1$ there exists a bounded (linear) Jordan homomorphism $\Phi_{a,b}: \mathcal{J}_1\to \mathcal{J}_2$, depending on $a$ and $b$, such that $\Phi_{a,b} (a) = T(a)$ and $\Phi_{a,b}(b) = T(b)$. When $\mathcal{J}_1$ and $\mathcal{J}_2$ are JB$^*$-algebras, a mapping $T: \mathcal{J}_1\to \mathcal{J}_2$ is called a \emph{2-local Jordan $^*$-homomorphism} if for every $a,b\in \mathcal{J}_1$ there exists a Jordan $^*$-homomorphism $\Phi_{a,b}: \mathcal{J}_1\to \mathcal{J}_2$, depending on $a$ and $b$, such that $\Phi_{a,b} (a) = T(a)$ and $\Phi_{a,b}(b) = T(b)$.\smallskip

The statements $(a),$ $(b)$ and $(d)-(h)$ in Lemma \ref{l basic properties} remain valid for 2-local Jordan homomorphisms between Jordan Banach algebras and for 2-local Jordan $^*$-homomorphisms between JB$^*$-algebras, with the particularity that elements $a$ and $b$ in a JB$^*$-algebra $\mathcal{J}$ are \emph{orthogonal} ($a\perp b$) if and only if $(a \circ b^*) \circ x + (x \circ b^*) \circ a - (a\circ x) \circ b^* = 0$ for every $x\in \mathcal{J}$, or equivalently $(a \circ a^*) \circ b + (b \circ a^*) \circ a - (a\circ b) \circ a^* = 0$ (see \cite[Lemma 1.1]{BurFerGarMarPe}). Clearly, when $a$ is a self-adjoint element $a\perp b$ if and only if $a^2\circ b =0.$ Further, if $a$ is positive, $a\perp b$ if and only if $a\circ b=0$ (cf. \cite[Lemma 4.1]{BurFerGarPe09}).\smallskip

Let $T: \mathcal{J}_1\to \mathcal{J}_2$ be a (not necessarily linear) 2-local Jordan $^*$-homomor-phism between JB$^*$-algebras and let $p_1,\ldots, p_n$ be a family of mutually orthogonal projections in $\mathcal{J}_1$. The arguments in the proofs of Lemmas \ref{l additivity on orthogonal idempotents} and \ref{l linearity on orthogonal idempotents} can be slightly modified to be valid in the Jordan setting in order to prove that  \begin{equation}\label{eq linearity Jordan *-homomoprhisms}  T\left(\sum_{i=1}^{n} \lambda_i p_i \right) = \sum_{i=1}^{n} \lambda_i T(p_i),
 \end{equation}for every $\lambda_1, \ldots, \lambda_n\in \mathbb{C}$. Furthermore, Lemmas \ref{l complex combinations hermitians} and \ref{l linearity on single generated subalgebras} are true for 2-local Jordan $^*$-homomorphisms between JB$^*$-algebras, and thus \begin{enumerate}[$(4.i)$] \item $T (a + i b) = T(a) + i T(b),$ for every $a,b\in \left(\mathcal{J}_{1}\right)_{sa}$;
 \item For each $h\in \left(\mathcal{J}_{1}\right)_{sa}$, $T_{\left(\mathcal{J}_{1}\right)_{h}} : \left(\mathcal{J}_{1}\right)_{h} \to \mathcal{J}_{2}$ is a linear mapping, where $\left(\mathcal{J}_{1}\right)_{h}$ denotes the JB$^*$-subalgebra generated by the element $h$.
 \end{enumerate}

When in the proof of Theorem \ref{t 2-local hom on von Neumann algebras not containing M_2} we replace the Bunce-Wright-Mackey-Gleason theorem for von Neumann algebras (cf. \cite[Theorem A]{BuWri92} or \cite{BuWri94}) with an appropriate version for JBW-algebras (see \cite[Theorem 2.1]{BuWri89}) we obtain the following:

\begin{theorem}\label{t 2-local Jordan hom on von Neumann algebras not containing I_2} Let $\mathcal{J}$ be a JBW$^*$-algebra without Type $I_2$ part and let $\mathcal{B}$ be a JB$^*$-algebra. Suppose $T : \mathcal{J} \to \mathcal{B}$ is a (not necessarily linear) 2-local Jordan $^*$-homomorphism. Then $T$ is linear and a Jordan $^*$-homomorphism. $\hfill\Box$
\end{theorem}

We recall now a result which is part of the folklore in C$^*$-algebra theory: suppose $p$ is a projection in a unital C$^*$-algebra $A$ and $x$ is a norm-one (self-adjoint) element in $A$ satisfying that $p x p = p$ then \begin{equation}\label{eq FriRu 1.6 C*}
x= p + (1-p) x (1-p).
\end{equation} Indeed, since $1\geq \| p x \|$ and $$p x = p + p x (1-p),$$ it follows that $$(p x) (p x)^* = p + p x (1-p) x^* p,$$ is a positive norm-one element in $p A p$. Moreover, since $\|(px) (px)^*\|\leq 1,$ and the element $p x (1-p) x^* p$ also is positive in $p A p,$ it must be zero, and hence $p x (1-p) = 0$. We similarly get $(1-p) x p=0$ and the desired statement \eqref{eq FriRu 1.6 C*}.\smallskip

Now, let $p$ be a projection in a unital JB$^*$-algebra $\mathcal{J}$ and let $x$ be a norm-one self-adjoint element in $\mathcal{J}$ satisfying that $U_{p} (x) = p$ then \begin{equation}\label{eq FriRu 1.6 JB*}
x= p + U_{(1-p)} (x).
\end{equation} In effect, by the Shirshov-Cohn theorem \cite[2.4.14]{HOS}, the JB$^*$-subalgebra $\mathcal{B}$ of $\mathcal{J}$ generated by $p$, $x$ and the unit element is special, that is, there exists a C$^*$-algebra $A$ such that $\mathcal{B}$ is a JB$^*$-subalgebra of $A$. In such a case, $p$ is a projection in $\mathcal{B}$, and hence in $A$, and it is satisfied that $\|x\|=1$, and  $p x p = U_{p} (x) = p$. We deduce from \eqref{eq FriRu 1.6 C*} that $x=p + (1-p) x (1-p)= p + U_{(1-p)} (x).$\smallskip

A stronger version of \eqref{eq FriRu 1.6 C*} and \eqref{eq FriRu 1.6 JB*} was established by Friedman and Russo in \cite[Lemma 1.6]{FriRu85} for tripotents in a JB$^*$-triple.\smallskip

To deal with JBW$^*$-algebras of Type $I_2$, we shall need the following partial result.

\begin{proposition}\label{p 2-local Jordan *-hom on H2} Let $\mathcal{J}$ a JB-algebra, and let $T: M_2(\mathbb{R})_{sa}\to \mathcal{J}$ be a (not necessarily linear) 2-local Jordan homomorphism. Then $T$ is linear and a Jordan homomorphism.
\end{proposition}

\begin{proof} Replacing $\mathcal{J}$ with $\mathcal{J}^{**}$ we can always assume that $\mathcal{J}$ is unital.
By Lemma \ref{l 2-local on factors}, whose statement remains valid for JBW-algebra factors, we may assume that $T$ is a 2-local Jordan monomorphism.\smallskip

Let us denote $u= \left(
                                                         \begin{array}{cc}
                                                           0 & 1 \\
                                                           1 & 0 \\
                                                         \end{array}
                                                       \right)$,  $p_1 = \left(
                                                                                           \begin{array}{cc}
                                                                                             1 & 0 \\
                                                                                             0 & 0 \\
                                                                                           \end{array}
                                                                                         \right)$, and $p_2= \left(
                                                                                                           \begin{array}{cc}
                                                                                                             0 & 0 \\
                                                                                                             0 & 1 \\
                                                                                                           \end{array}
                                                                                                         \right)
                                                                                         $.
Clearly, $T(p_1)$ and $T(p_2)$ are mutually orthogonal projections, $T(u)^3 = T(u)$, $T(p_i) \circ T(u) = \frac12 T(u)$ (and hence, $U_{T(p_1)} (T(u)) =0= U_{1-T(p_1)} (T(u)) $), and $T(u)^2 = T(1)$.\smallskip


First we prove that \begin{equation}\label{eq one in Lema H2} T(\lambda p_1 + \mu u ) = \lambda T(p_1) + \mu T(u),
\end{equation} for every $\lambda, \mu \in \mathbb{R}$. For the proof we assume $\lambda,\mu \neq 0$. Indeed, set $z= \lambda p_1 + \mu u$ and consider the Jordan homomorphisms $\Phi_{z,p_1}, \Phi_{u,z}: M_2(\mathbb{R})_{sa}\to \mathcal{J}$. The equalities $$T(z) = \lambda \Phi_{z,p_1} (p_1) + \mu \Phi_{z,p_1} (u)=  \lambda T(p_1) + \mu \Phi_{z,p_1} (u),$$
$$T(z) = \lambda \Phi_{z,u} (p_1) + \mu \Phi_{z,u} (u) = \lambda \Phi_{z,u} (p_1) + \mu T (u),$$
imply that $$\Phi_{z,u} (p_1) = T(p_1) +\frac{\mu}{\lambda} (\Phi_{z,p_1} (u) - T(u)).$$ The elements $\Phi_{z,u} (p_1)$ and  $T(p_1)$ are norm-one projections. Since $U_{p_1} (u) =0= U_{1-p_1} (u) $, $$U_{T(p_1)} (T(u)) =0= U_{1-T(p_1)} (T(u)),$$ and $$U_{T(p_1)} (\Phi_{z,p_1}(u)) =U_{\Phi_{z,p_1}(p_1)} (\Phi_{z,p_1}(u)) =0,$$ $$0= U_{1-\Phi_{z,p_1}(p_1)} (\Phi_{z,p_1}(u)) = U_{1-T(p_1)} (\Phi_{z,p_1}(u)),$$ we deduce from \eqref{eq FriRu 1.6 JB*} that $$\frac{\mu}{\lambda} (\Phi_{z,p_1} (u) - T(u))=0,$$ which gives $\Phi_{z,p_1} (u) = T(u),$ and hence $T(\lambda p_1 + \mu u ) = \lambda T(p_1) + \mu T(u).$\smallskip

Similarly, we show that \begin{equation}\label{eq two in Lema H2} T(\lambda p_2 + \mu u ) = \lambda T(p_2) + \mu T(u),
\end{equation} for every $\lambda, \mu \in \mathbb{R}$.\smallskip

To conclude the proof we shall prove that $$T(\lambda p_1 + \mu u + \gamma p_2) = \lambda T(p_1) + \mu T(u) + \gamma T(p_2),$$ for every $\lambda, \mu, \gamma \in \mathbb{R}$. Pick $\lambda, \mu, \gamma \in \mathbb{R}\backslash \{0\}$ and set $z=\lambda p_1 + \mu u + \gamma p_2$. Applying \eqref{eq one in Lema H2} we get \begin{equation}\label{eq *1}T(z) = \Phi_{z,\lambda p_1 + \mu u} (z) = \Phi_{z,\lambda p_1 + \mu u} (\lambda p_1 + \mu u ) + \gamma \Phi_{z,\lambda p_1 + \mu u} ( p_2)
 \end{equation} $$ =  T (\lambda p_1 + \mu u ) + \gamma \Phi_{z,\lambda p_1 + \mu u} ( p_2)  = \lambda T (p_1) + \mu T (u ) + \gamma \Phi_{z,\lambda p_1 + \mu u} ( p_2).$$ And on the other hand,
\begin{equation}\label{eq *2}T(z) = \Phi_{z,p_2} (z) = \lambda \Phi_{z, p_2} ( p_1 )  + \mu \Phi_{z, p_2} ( u ) + \gamma T ( p_2).
 \end{equation} Combining \eqref{eq *1} and \eqref{eq *2} we get $$\Phi_{z,\lambda p_1 + \mu u} ( p_2) = T(p_2) +\frac{1}{\gamma} (\lambda \Phi_{z, p_2} ( p_1 ) - \lambda T(p_1) ) +\frac{\mu}{\gamma} (\Phi_{z, p_2} ( u )- T(u)).$$ Having in mind that $\Phi_{z,\lambda p_1 + \mu u} ( p_2)$ and $ T(p_2)$ are norm-one projections, $$U_{T(p_2)} (T(u)) =0= U_{1-T(p_2)} (T(u)),$$ $$U_{T(p_2)} (\Phi_{z,p_2}(u)) =U_{\Phi_{z,p_2}(p_2)} (\Phi_{z,p_2}(u)) =0,$$ $$U_{1-T(p_2)} (\Phi_{z,p_2}(u))= U_{1-\Phi_{z,p_2}(p_2)} (\Phi_{z,p_2}(u)) = 0,$$ by orthogonality, $$U_{1-T(p_2)} (T(p_1))= T(p_1),$$ and $$ U_{1-T(p_2)} (\Phi_{z, p_2}(p_1))= U_{1-\Phi_{z, p_2}(p_2)} (\Phi_{z, p_2}(p_1))= \Phi_{z, p_2}(p_1),$$ we deduce from \eqref{eq FriRu 1.6 JB*} that $$\Phi_{z, p_2} ( u )= T(u).$$ Therefore, \eqref{eq *2} writes in the form \begin{equation}\label{eq *2 new} T(z) = \lambda \Phi_{z, p_2} ( p_1 )  + \mu T ( u ) + \gamma T ( p_2).
 \end{equation}

Now, applying \eqref{eq two in Lema H2} we get \begin{equation}\label{eq *4}T(z) = \Phi_{z,\gamma p_2 + \mu u} (z) = \lambda \Phi_{z,\gamma p_2 + \mu u} (p_1) + \mu T (u ) + \gamma T ( p_2).
 \end{equation} Independently,
\begin{equation}\label{eq *3} T(z) = \Phi_{z,p_1} (z) = \lambda T ( p_1 )  + \mu \Phi_{z, p_1} ( u ) + \gamma \Phi_{z, p_1} ( p_2).
 \end{equation} Arguing as in the previous paragraphs, we deduce from these identities that $\Phi_{z, p_1} ( u ) = T(u)$, and thus \eqref{eq *3} writes in the form \begin{equation}\label{eq *3 new} T(z) = \lambda T ( p_1 )  + \mu T ( u ) + \gamma \Phi_{z, p_1} ( p_2).
 \end{equation}

Combining \eqref{eq *2 new} and \eqref{eq *3 new} we obtain: $$ \Phi_{z, p_1} ( p_2) =  T ( p_2) + \frac{\lambda}{ \gamma} (\Phi_{z, p_2} ( p_1 )  -  T ( p_1 )).$$ Since $\Phi_{z, p_2} ( p_1) \perp \Phi_{z, p_2} ( p_2) = T(p_2)$, $T(p_2) \perp T(p_1)$, the projection $\Phi_{z, p_1} ( p_2)$ is equal to the orthogonal sum of the projection $T(p_2)$ and the element $ \frac{\lambda}{ \gamma} (\Phi_{z, p_2} ( p_1 )  -  T ( p_1 ))$. It follows that $\frac{\lambda}{ \gamma} (\Phi_{z, p_2} ( p_1 )  -  T ( p_1 ))$ is a projection orthogonal to $T(p_2)$. However, $\Phi_{z, p_1} ( p_2) \perp \Phi_{z, p_1} ( p_1) = T(p_1)$ assures that
$$ \frac{\lambda}{ \gamma} (\Phi_{z, p_2} ( p_1 )  -  T ( p_1 )) =  (\Phi_{z, p_1} ( p_2) -  T ( p_2)) \perp T(p_1).$$ That is, \begin{equation}\label{eq *5 perp} (\Phi_{z, p_2} ( p_1 )  -  T ( p_1 )) \perp \left(T(p_2) + T(p_1) \right)=T(p_1+p_2) = T(1)
  \end{equation}(compare \eqref{eq linearity Jordan *-homomoprhisms}).\smallskip

It is known that for each $a\in M_2 (\mathbb{R})_{sa}$, $U_{T(1)} T(a) = U_{\Phi_{a,1}(1)} \Phi_{a,1} (a) = \Phi_{a,1} U_1 (a) = T(a)$. In particular, the identity \eqref{eq *2 new} implies that $$U_{T(1)} (\Phi_{z, p_2} ( p_1 ) ) = \Phi_{z, p_2} ( p_1 ),$$ and consequently $$ U_{T(1)} (\Phi_{z, p_2} ( p_1 ) -  T ( p_1 ) ) = \Phi_{z, p_2} ( p_1 ) -  T ( p_1 ).$$ This identity together with \eqref{eq *5 perp} show that $\Phi_{z, p_2} ( p_1 ) -  T ( p_1 )=0$. Finally, the identity \eqref{eq *2 new} writes in the form  $$T(\lambda p_1 + \mu u + \gamma p_2) = \lambda T ( p_1 )  + \mu T ( u ) + \gamma T ( p_2),$$ which proves the linearity of $T$.
\end{proof}

Let $\mathcal{J}_1$ and $\mathcal{J}_2$ be two JBW-algebras satisfying that, for each JB-algebra $\mathcal{B}$, every 2-local Jordan homomorphism $S: \mathcal{J}_i \to B$ is linear. Suppose $T: \mathcal{J}_1\oplus^{\infty} \mathcal{J}_2 \to B$ is a 2-local Jordan homomorphism. Given $a\in\mathcal{J}_1$, $b \in \mathcal{J}_2$, and $\varepsilon>0$, we can find two families $p_1,\ldots,p_m$ and $q_1,\ldots,q_k$ of mutually orthogonal projections in $\mathcal{J}_1$ and  $\mathcal{J}_2$, respectively, and real numbers $\lambda_i$, $\mu_j$ such that $$\left\|a-\sum_{i=1}\lambda_i p_i \right\| <\frac{\varepsilon}{4}, \hbox{ and } \left\|b-\sum_{j=1}\mu_j q_j \right\| <\frac{\varepsilon}{4}.$$ We note that $p_i\perp q_j$, for every $i,j$. In this case, by \eqref{eq linearity Jordan *-homomoprhisms}, $$\left\|T(a +b) -T(a)-T(b) \right\|  \leq \left\|T(a +b) -T\left(\sum_{i=1}\lambda_i p_i+ \sum_{j=1}\mu_j q_j \right) \right\| $$
$$+\left\|T\left(\sum_{i=1}\lambda_i p_i \right)-T(a) \right\|+ \left\|T\left(\sum_{j=1}\mu_j q_j \right)-T(b) \right\|< \varepsilon.$$ Since $\varepsilon$ was arbitrarily chosen, the assumptions on $\mathcal{J}_1$ and $\mathcal{J}_2$ show that $T$ is linear.\smallskip

A induction argument shows:

\begin{corollary}\label{c stability for ell infty sums JBW} Let $\left(\mathcal{J}_i\right)_{i=1}^{n}$ be a finite family of JBW-algebras and let $\mathcal{B}$ be a JB-algebra. Suppose that, for every $i$, every 2-local Jordan homomorphism $T: \mathcal{J}_i \to \mathcal{B}$ is linear. Then every 2-local Jordan homomorphism $\displaystyle T: \bigoplus_{i=1,\ldots,n}^{\ell_{\infty}} \mathcal{J}_i \to \mathcal{B}$ is linear. $\hfill \Box$
\end{corollary}

We can establish now the Jordan version of Corollary \ref{c type I2}.

\begin{corollary}\label{c type I2 Jordan} Every (not necessarily linear) 2-local Jordan $^*$-homomor-phism from a Type $I_2$ JBW$^*$-algebra into a JB$^*$-algebra is linear and a Jordan $^*$-homomorphism.
\end{corollary}

\begin{proof} Having in mind $(4.i)$, we observe that it is enough to prove that $T: \mathcal{J}_{sa} \to \mathcal{B}_{sa}$ is additive. We follow some of the arguments given by Bunce and Hamhalter in \cite[Theorem in page 158]{BuHam}. Let $x,y\in \mathcal{J}_{sa}$. As noted by the authors in \cite[proof of Theorem in page 158]{BuHam}, the JBW-subalgebra, $\mathcal{J}_{sa} (x,y)$, of $\mathcal{J}_{sa}$ generated by $x$ and $y$ identifies, via \cite[Theorem 2]{Sta82}, $$ \mathcal{J}_{sa}(x,y) = C(K,\mathbb{R}) \otimes M_2(\mathbb{R})_{sa},$$ where $C(K,\mathbb{R})$ is the algebra of all continuous real-valued functions on a compact Stonean space $K$.\smallskip

When in the proof of Corollary \ref{c type I2}, Proposition \ref{p 2-local *-hom on M2} and Corollary \ref{c stability for ell infty sums} are replaced with Proposition \ref{p 2-local Jordan *-hom on H2} and Corollary \ref{c stability for ell infty sums JBW}, respectively, the arguments remain valid to prove the desired statement.
\end{proof}

Combining Theorem \ref{t 2-local Jordan hom on von Neumann algebras not containing I_2}, Corollaries \ref{c type I2 Jordan} and \ref{c stability for ell infty sums JBW} and the structure theory of JBW$^*$-algebras (cf. \cite[\S 5]{HOS}) we obtain the following Jordan version of Theorem \ref{t 2-local hom on von Neumann algebras}.

\begin{theorem}\label{t 2-local Jordan hom on JBW-algebras} Every (not necessarily linear) 2-local Jordan $^*$-homomor-phism from a JBW$^*$-algebra into a JB$^*$-algebra is linear and a Jordan $^*$-homomorphism.$\hfill\Box$
\end{theorem}

\begin{corollary}\label{c 2-local Jordan hom on von Neumann algebras} Every (not necessarily linear) 2-local Jordan $^*$-homomor-phism from a von Neumann algebra into a C$^*$-algebra is linear and a Jordan $^*$-homomorphism.$\hfill\Box$
\end{corollary}

\end{document}